\documentclass[openany,10pt,letter]{article}
\usepackage{todonotes}
\usepackage[utf8]{inputenc}
\usepackage[sumlimits]{amsmath}
\usepackage{amsfonts,amssymb,amsthm}
\usepackage[english]{babel}

\numberwithin{equation}{section}
\newtheorem{theorem}{Theorem}[section]
\newtheorem{lemma}[theorem]{Lemma}
\newtheorem{proposition}[theorem]{Proposition}

\theoremstyle{definition}
\newtheorem{definition}[theorem]{Definition}
\newtheorem{example}[theorem]{Example}

\newtheorem{obs}[theorem]{Observation}
\newtheorem{remark}[theorem]{Remark}

\newcommand{\newword}[1]{\textbf{#1}}

\newcommand{\Sol}{\operatorname{Sol}}
\newcommand{\trop}{\operatorname{trop}}
\newcommand{\In}{\operatorname{In}}
\newcommand{\Val}{\operatorname{Val}}
\newcommand{\supp}{\operatorname{supp}}
\newcommand{\suppmin}{\mathop{\operatorname{supp}_{\min}}}

\title{Initial forms and a notion of basis for tropical differential equations}
\author{Alex Fink and Zeinab Toghani}

\begin{document}
\maketitle

\begin{abstract}
We show that solution sets of systems of tropical differential equations can be characterised in terms of monomial-freeness of an initial ideal.
We discuss a candidate definition of tropical differential basis and give a nonexistence result for such bases in an example.
\end{abstract}


\section{Introduction}
In \cite{Grigoriev:2015}, Grigoriev introduced a tropical approach to differential equations.
He was interested in constraining the supports of power series solutions
to a system of differential equations in characteristic zero
in an effective algorithmic way,
and found that tropical techniques give rise to some limitations on these supports.

Aroca, Garay and Toghani \cite{ArocaGarayToghani:2016} extended this work,
bringing to the subject the point of view that tropical solution sets of differential systems
can be thought of as analogues to tropical varieties.
Their main result was a Fundamental Theorem of Tropical Differential Geometry (the equivalence (\ref{set1})$\Leftrightarrow$(\ref{set2}) of Theorem~\ref{pp} below):
a system of differential equations over a field~$K$ can be tropicalised, and under mild assumptions, the solutions of this tropicalisation
agree with the tropicalisations of the $K$-valued solutions in Grigoriev's sense.

In the non-differential setting of ideals in a polynomial ring, the heart of the Fundamental Theorem of Tropical Geometry
is the statement that taking the solution set commutes with tropicalisation.
But further equivalences can be added to the theorem:
notably, the statement in \cite[Theorem 3.2.3]{MaclaganSturmfels:2015} includes a description in terms of initial ideals.
In this paper, we add an initial ideal characterisation to the Fundamental Theorem of Tropical Differential Geometry (Theorem~\ref{pp}).

The datum needed to define an initial ideal of a differential polynomial $f$ in Grigoriev's formalism
is a set or tuple of sets $S$ of natural numbers. 
A potential solution to $f$ is a datum $S$ of the same kind, which represents the supports of power series.
In both respects $S$ takes the place of the weight vector in~$\mathbb R^n$ used in tropical (algebraic) geometry.
In a 2019 preprint \cite{HuGao2019}, Hu and Gao define the $S$-initial part of a differential ideal with coefficients in the power series ring $K[[t]]$, denoting it $\operatorname{in}_S(P)$.
We arrived at the definition independently, and as such we formulate it in a mildly different way:
in brief, our initial ideal $\In_S(P)$ has coefficients in the residue field $K$ of~$K[[t]]$, 
whereas Hu and Gao retain powers of~$t$ in the coefficients of their $\operatorname{in}_S(P)$.
To distinguish the two we denote our formulation as~$\In_S(P)$, with a capital letter.
See Remark~\ref{rem:HuGao} for a detailed comparison.

Tropical bases are a central concept in tropical geometry, especially for computational approaches.
An ideal in a polynomial ring has a finite tropical basis.
The subject of Section~\ref{sec:bases} of this paper 
is a notion of tropical differential basis, Definition~\ref{def:basis},
which, although it might look initially compelling, does not allow a comparably good finiteness result.
To be precise, we exhibit linear differential ideals with no finite basis of linear forms in this sense.
In the non-differential theory, tropical linear spaces have finite tropical bases of linear forms,
arising from circuits of valuated matroids \cite[Sec.~4.4]{MaclaganSturmfels:2015}.
Given how much simpler the theory of linear differential equations is than the general theory,
a definition of basis which has to leave the linear world must be considered lacking.
Thus, we leave it as an open question to find a better definition of basis for tropical differential ideals.

The main thrust of the work \cite{HuGao2019} is to introduce a notion of Gr\"obner basis for tropical differential ideals.
Tropical differential Gr\"obner bases are usually infinite.
But, just as there is no implicational relationship between tropical bases and (universal) Gr\"obner bases for ideals in a polynomial ring,
it is unclear whether tropical differential Gr\"obner bases are of utility for answering our open question.

\paragraph{Acknowledgments}
This work has received funding from the European Union's Horizon 2020 research and innovation programme under the Marie Sk\l odowska-Curie grant agreement No~792432.
The authors would like to thank Cristhian Garay L\'opez, Jeff Giansiracusa, Dima Grigoriev, Yue Ren, and Felipe Rinc\'on for helpful conversations during its preparation.

\section{Background}
In this section we review definitions.
Our notation is compatible with that of~\cite{ArocaGarayToghani:2016},
although we deviate by introducing the multi-index notation.

\subsection{Differential algebra}
 Let $R$ be a commutative ring with unity. A \newword{derivation} on $R$ is a map $d:R\to R$ that satisfies   $d(a+b)=d(a)+d(b)$ and $d(ab)=d(a)b+ad(b)$ for all $ a,b\in R$. The pair $(R,d)$ is called a \newword{differential ring}. An ideal $I\subset R$ is said to be a \newword{differential ideal} when $d(I)\subset I$.

Let $(R,d)$ be a  differential ring and let $R\{x_1,\ldots,x_n\}$ be the set of polynomials with coefficients in $R$ in the (\newword{differential}) variables $\{x_{ij}\::\:i=1,\ldots,n,\:j\geq0\}$. 
An element of~$R$ is called a \newword{differential polynomial}.
The derivation $d$ on $R$ can be extended to a derivation $d$ of $R\{x_1,\ldots,x_n\}$ by setting $d(x_{ij})=x_{i(j+1)}$ for  $i=1,\dots,n$ and $j\geq0$. The pair $(R\{x_1,\ldots,x_n\},d)$ is a differential ring called \newword{the ring of differential polynomials in $n$ variables with coefficients in $R$}.

A differential polynomial $P\in R\{x_1,\ldots,x_n\}$ can be evaluated at an $n$-tuple $\varphi=(\varphi_1,\ldots,\varphi_n)\in R^n$ 
by evaluating each differential variable $x_{ij}$ at $d^j\varphi_i$.
 A \newword{zero} or a \newword{solution} of~$P$ is an $n$-tuple $\varphi\in R^n$ such that $P(\varphi)=0$. An $n$-tuple $\varphi\in R^n$ is a \newword{solution} of $A\subset R\{x_1,\ldots,x_n\}$ when it is a solution of every differential polynomial in $\Sigma$. That is,
\[
	\Sol (A):=\{\varphi\in R^n : P(\varphi)=0,\forall P\in A\}.
\]

A \newword{differential monomial} in $R\{x_1,\ldots,x_n\}$ of \newword{order} less than or equal to $r$ is an expression of the form
\begin{equation*}
	\varphi_M\prod_{\substack{1\leq i\leq n\\ 0\leq j\leq r}}x_{ij}^{M_{ij}},
\end{equation*}
where $M=(M_{ij})_{\substack{1\leq i\leq n\\ 0\leq j\leq r}}$ is a matrix in $\mathcal{M}_{n\times(r+1)}(\mathbb{Z}_{\geq0})$, $r\in \mathbb{Z}_{\geq0}$, and $\varphi_M\in R$.
We will use multi-index notation and abbreviate the product $\prod_{1\leq i\leq n,0\leq j\leq r}x_{ij}^{M_{ij}}$ as $x^M$,
so the differential monomial above can be written $\varphi_Mx^M$.
A differential polynomial in $R\{x_1,\ldots,x_n\}$ has order less than or equal to $r$ if it is of the form
\begin{equation}\label{pdif2}
	P=\sum_{M\in\Lambda}\varphi_M x^M,
\end{equation}
with $\Lambda\subset \mathcal{M}_{n\times(r	+1)}(\mathbb{Z}_{\geq0})$ finite.

\subsection{The power series ring} 
In what follows, we will work with the power series ring $K[[t]]$ where $K$ is a field of characteristic zero.  
We fix the structure of a differential valued ring on $K[[t]]$ as follows.
We can write the elements of $K[[t]]$ in the form
\begin{equation*}
	\varphi = \sum_{i\in\mathbb{Z}_{\geq0}} a_i t^i
\end{equation*}
with $a_i\in K$ for $i\in\mathbb{Z}_{\geq0}$.
Given such an element $\varphi$,
the \newword{support}  of $\varphi$ is the set 
\[
\supp(\varphi) :=\{i\in\mathbb{Z}_{\geq0}\::\:a_i\neq0\},
\]
Then the valuation on $K[[t]]$ is given by
 \[
\nu (\varphi ) = \min \supp(\varphi)
\]
and the derivation by
\begin{displaymath}
\begin{array}{cccc}
 d:& K[[t]]& \rightarrow & K[[t]]\\
 & \varphi & \mapsto  &\sum_{i\in\mathbb{Z}_{\geq 0}} ia_{i} t^{i-1}.
\end{array}
\end{displaymath} 
Iterating, we see that the $j$-th derivative of $\varphi$ is
\[
d^j\varphi = \sum_{i\in\mathbb{Z}_{\geq 0}}\frac{1}{i!} a_{i+j} t^{i}.
\]

The mapping that sends each series in $K[[t]]$ to its support set (a subset of $\mathbb{Z}_{\geq 0}$) will be called the \newword{tropicalisation} map
\[
	\begin{array}{cccc}
		\text{trop}:
			& K[[t]]
				&\to
					& \mathcal{P}(\mathbb{Z}_{\geq 0})\\
			& \varphi
				&\mapsto
					&\supp(\varphi)
	\end{array}
\]
where $\mathcal{P}(\mathbb{Z}_{\geq0})$ denotes the power set of $\mathbb{Z}_{\geq0}$.
For fixed $n$, the mapping from $K[[t]]^n$ to the $n$-fold product of $\mathcal{P}(\mathbb{Z}_{\geq 0})$ will also be denoted by $\text{trop}$:
\[
	\begin{array}{cccc}
		\trop:
			& K[[t]]^n
				&\to
					& {\mathcal{P}(\mathbb{Z}_{\geq 0})}^n\\
			& \varphi=(\varphi_1,\ldots ,\varphi_n)
				&\mapsto
					& \trop(\varphi)=(\supp(\varphi_1),\ldots ,\supp(\varphi_n)).
	\end{array}
\]

Given a subset $T$ of $K[[t]]^n$, the \newword{tropicalisation} of~$T$ is its image under the map $\trop$:
\[
\text{trop} (T):=\left\{ \trop(\varphi) \::\: \varphi\in T\right\}\subset {\mathcal{P}(\mathbb{Z}_{\geq 0})}^n.
\]

\begin{example}
We consider $\varphi=(\varphi_{1},\varphi_{2},\varphi_{3})=(a+bt, t^3, t+t^2)\subset {K[[t]]}^3$. We have
\[
\begin{array}{ll}
	\text{trop}(\varphi) = 
		&\{(\{0\}, \{ 3\},\{1,2\}), (\{1\}, \{ 3\},\{1,2\}), (\{0,1\}, \{ 3\},\{1,2\}) \}.
\end{array}
\]
\end{example}

The reader will note that the map $\trop$ remembers more information about a power series
than the non-archimedean valuation $\nu$ typically used in tropical mathematics.
The reason for this is to allow the derivation to have a well-defined action on the target of $\trop$.
To wit, since $K$ is of characteristic zero, for every $\varphi\in K[[t]]$, we have
\[
\text{trop} \left( d^j\varphi\right) = \left\{ i-j : i\in \text{trop} (\varphi)\cap \mathbb{Z}_{\geq j}\right\}
\]
and so
\[
\nu\left( d^j\varphi\right) = \min \left(\text{trop} (\varphi)\cap \mathbb{Z}_{\geq j}\right) -j.
\]
Therefore $\trop(\varphi)$ determines the valuation of the derivatives of $\varphi$ of all orders.
Our notation for these valuations is as follows.
\begin{definition}				
A subset $S\subseteq\mathbb{Z}_{\geq0}$ induces a mapping $\Val_S:\mathbb{Z}_{\geq 0}\to\mathbb{Z}_{\geq0}\cup\{\infty\}$ given by 
\begin{equation}\label{tropSol2}
\Val_S(j):= \begin{cases}
s-j,&\text{ with }s=\text{min}\{\alpha\in S\::\:\alpha\geq j\},\\
\infty,&\text{ when } S\cap \mathbb{Z}_{\geq j}=\emptyset.
\end{cases}
\end{equation}
\end{definition}

\begin{example}
Consider the set $S:=\{0,1,2,3,7,8\}$. We have
$\Val_{S}(4)=\min \{s\in S \mid s\ge 4\}-4=7-4=3$,
while $\Val_{S}(9)=\infty$.
\end{example}

\subsection{Tropical differential algebra}
We will denote by $\mathbb{T}_{\ge0}$ the semiring $\mathbb{T}_{\ge0}=(\mathbb{Z}_{\geq0}\cup\{\infty\},\oplus,\odot)$, with $a\oplus b=\text{min}\{a,b\}$ and $a\odot b=a+b$. 
Later, in Section~\ref{sec:bases}, 
we will also invoke the semiring $\mathbb{T} = (\mathbb{R}\cup\{\infty\},\oplus,\odot)$ with the same operations.

\begin{definition}
A \newword{tropical differential polynomial} in the variables $x_1,\ldots,x_n$ of order less than or equal to $r$ is an expression of the form 
\begin{equation}\label{tdp}
\varphi=\varphi(x_1,\ldots,x_n)=\underset{M\in\Lambda}{\bigoplus} a_{M}x^{\odot M},
\end{equation}
where $M=(M_{ij})_{\substack{1\leq i\leq n\\ 0\leq j\leq r}}$ is a matrix in $\mathcal{M}_{n\times(r+1)}(\mathbb{Z}_{\geq 0})$, $a_{M}\in\mathbb{T}_{\ge0}$ and $\Lambda\subset \mathcal{M}_{n\times(r+1)}(\mathbb{Z}_{\geq0})$ is a finite set.
\end{definition}
Again the multi-index notation $x^{\odot M}$ stands for $\bigodot_{\substack{1\leq i\leq n\\ 0\leq j\leq r}}{x_{ij}}^{\odot M_{ij}}$.
The set of tropical differential polynomials in $ x_{1},\ldots ,x_{n} $ will be denoted by $\mathbb{T}\{x_1,\ldots,x_n\}$. 

A tropical differential polynomial $\varphi$ as in~\eqref{tdp} induces an evaluation mapping from $\mathcal{P}(\mathbb{Z}_{\geq0})^n$ to $\mathbb{Z}_{\geq0}\cup\{\infty\}$
given by 
 \[
 \varphi(S)=\underset{M\in\Lambda}{\text{min}}\{a_{M}+\sum_{\substack{1\leq i\leq n\\ 0\leq j\leq r}}M_{ij}\cdot \Val_{S_i}(j)\}
\]
where $\Val_{S_i}(j)$ is defined as in~\eqref{tropSol2}.

\begin{definition}\label{def:tropical solution}
An $n$-tuple $S=(S_1,\ldots,S_n)\in \mathcal{P}(\mathbb{Z}_{\geq0})^n$ is said to be a \newword{solution} of the tropical differential polynomial $\varphi$ in~\eqref{tdp} if either

\begin{enumerate}
\item there exist $M_1,M_2\in\Lambda$, $M_1\neq M_2$, such that $\varphi(S)=a_{M_1}\odot\varepsilon_{M_1}(S)=a_{M_2}\odot\varepsilon_{M_2}(S)$; or
\item $\varphi(S)=\infty$.
\end{enumerate}
\end{definition}
Let $A\subset\mathbb{T}_{\ge0}\{x_1,\ldots,x_n\}$ be a system of tropical differential polynomials. An $n$-tuple $S\in \mathcal{P}(\mathbb{Z}_{\geq0})^n$ is a \newword{solution} of $A$ when it is a solution of every tropical polynomial in $A$.
We denote the set of solutions of~$A$ by
\[
\Sol (A) := \left\{ S\in {(\mathcal{P}(\mathbb{Z}_{\geq0}))}^n : S \text{ is a solution of } \varphi \text{ for every }\varphi\in A\right\}.
\]

\begin{example}
Consider the tropical differential polynomial 
\[
\varphi=(x_{12})\oplus (2\odot x_{10})\oplus 1
\]

A set $S\in\mathcal{P}(\mathbb{Z}_{\geq 0})$ is a solution of $ \varphi $ if   
 $\min\{\Val_{S}(2) ,2+\Val_{S}(0), 1\}$ is attained at least twice. 
 This is equivalent, according to which pair of terms attain the minimum, to one of the following conditions being true.
 \begin{enumerate}
\item $\Val_S(2)= 2+\Val_S(0)\le 1$
\item $\Val_S(2)=1 \le 2+\Val_S(0)$
\item $2+\Val_S(0)=1 \le \Val_S(2)$.
\end{enumerate}
The first and third conditions do not hold for any~$S$, because $\Val_S(0)\ge0$.
The second condition holds if and only if
$2\not\in S$ and $3\in S$,
so just these sets $S$ are the solutions of~$\varphi$.
\end{example}

\subsection{Tropicalisation of differential polynomials} \label{Tropicalisation}
Let $P$ be a differential polynomial as in Equation~\eqref{pdif2}. The \newword{tropicalisation} of $P$ is the tropical differential polynomial 
\begin{equation*}
\text{trop}(P): = \underset{M\in\Lambda}{\bigoplus} \nu(\varphi_{M})x^{\odot M} .
\end{equation*}

\begin{definition}
Let $I\subset K[[t]]\{x_1,\ldots,x_n\}$ be a differential ideal. Its \newword{tropicalisation} is the set of tropical differential polynomials $\text{trop}(I)=\{ \text{trop}(P) : P\in I\}$.
\end{definition}

\begin{example}
We consider the differential polynomial $ P=tx_{12}^{3}x_{23}+(1+t^2)x_{13}^2$. Its tropicalisation is 
\[
\text{trop}(P)=\left( \nu(t)\odot x_{12}^{\odot 3}\odot x_{23}\right) \oplus  \left( \nu(1+t^2)\odot x_{13}^{\odot 2}\right) =\min\{1+3x_{12}+x_{23}, 2x_{13}  \}.
\] 
\end{example}

\section{The initial part of a differential ideal}
In this section we give definitions of $ S$-initial parts of differential polynomials and differential ideals.
We then prove our extension of the Fundamental Theorem of~\cite{ArocaGarayToghani:2016} to include a criterion on monomial-free $S$-initial parts (Theorem~\ref{pp}). 

Suppose that $Q$ is a differential polynomial,
$$ Q(x)=\sum_{M\in\Lambda }\psi_{M}x^M\in K((t))\{x_{1},\ldots,x_{n}\},$$
and that $ S =(S_{1},\ldots,S_{n})\subset \mathcal{P}(\mathbb{Z}_{\geq 0})^n$. We define the differential polynomial $ Q_{S}(x) $ by changing each variable $ x_{ij} $ to $ t^{\Val_{S_{i}}(j)}x_{ij} $ in $ Q$:
 %\in K((t))\{x_{1},\cdots, x_{n}\}. $$
 \begin{equation*}
   Q_{S}(x):=
  \begin{cases}
   t^{-\trop(Q)(S)}\,Q(t^{\Val_{S_{i}}(j)}x_{ij})_{i, j},&\text{ when } \trop(Q)(S)\neq \infty,\\
0,&\text{ when } \trop(Q)(S)=\infty.
\end{cases}
 \end{equation*}

 \begin{lemma}\label{newpolynomial}
  The polynomial $Q_{S}(x)$ is in the ring $K[[t]]\{x_1,\ldots,x_n\}$, 
  i.e., its coefficients have valuations are greater than or equal to zero. Also, $Q_{S}(x)$ has a coefficient of valuation zero.
\end{lemma}
\begin{proof}
   We can write
   \begin{align*}
\label{nnnn}
  Q_{S}(x) &= t^{-\trop(P)(S)}\,Q (t^{\Val_{S_{i}}(j)}x_{ij})_{\substack{ i,j}} 
  \\&=  t^{-\trop(Q)(S)}\sum_{M\in \Lambda} \psi_{M}\prod_{i,j} \big(t^{\Val_{S_{i}}(j)}x_{ij}\big)^{M_{ij}} 
  \\
   &=\sum_{M\in \Lambda} \psi_{M}t^{-\trop(Q)(S)+\sum M_{ij}\Val_{S_{i}}(j)}  x^M.
 \end{align*}
The valuation of the coefficient of $x^M$ is
\begin{align*}
\nu(t^{-\trop(Q)(S)+\sum M_{ij}\Val_{S_{i}}(j)})&=\\
-\trop(Q)(S)+\sum M_{ij}\Val_{S_{i}}(j) \geq- \nu(\psi_{M})+ \nu(\psi_{M})&=0.\qedhere
\end{align*}
\end{proof} 

Let $ \overline{Q_{S}}(x)$ be the image of $ Q_{S}$ in the ring $K\{x_1,\ldots,x_n\}$, under the map sending each coefficient $\psi\in K[[t]]$ to its image in the residue field~$K$. 
Then Lemma \ref{newpolynomial} implies
that $\overline{Q_{S}}$ is well-defined and nonzero.

 Let $P(x)=\sum_{M\in\Lambda }\varphi_{M}x^M$ be a differential polynomial with coefficients in the power series ring. The \newword{$ S$-initial part} of $ P $ is
 \begin{equation}\label{eq:In_S}
\In_{S}(P):= \begin{cases}
\overline{P_{S}}(x),&\text{ when } \trop(P)(S)\neq \infty,\\
0,&\text{ when } \trop(P)(S)=\infty.
\end{cases}
\end{equation}
In the first case we have 
  \begin{align*}
  \In_{S}(P)&=\overline{P_{S}}(x)
  \\&=\overline{\sum_{M\in\Lambda } \varphi_{M}t^{-\trop(P)(S)+\sum M_{ij}\Val_{S_{i}}} x^M}
    % \overline{\sum_{M\in\Lambda } \varphi_{M}  t^{-\trop(P)(S)+\sum M_{ij}\Val_{S_{i}}}x^M=
    \\&= \sum_{M\in \Upsilon}\overline{\varphi_M t^{-\nu(\varphi_{M})}x^M},
   \end{align*}
   where $\Upsilon= \{ M\in \Lambda \mid \trop(P)(S)= \nu(\varphi_{M})+\sum _{i,j}M_{ij}\Val_{S_{i}}(j)\}$.
\begin{example}
We exemplify the second case of equation~\eqref{eq:In_S}.
Consider the differential polynomial $$ P(x)=tx_{14}+t^2x_{15}, $$
 and take $S=\{1,2,3\}$. We have $\trop(P)(S)=\infty$, so $ \In_{S}(P)=0$.
\end{example}

\begin{remark}\label{rem:HuGao}
Consider a differential polynomial $P$ and a tuple $S\in (\mathbb{P}(\mathbb{Z}_{\geq 0}))^n$.  Let $\In_S(P)$ be the $S$-initial part of $P$ as defined above, and $\operatorname{in}_S(P)$ the $S$-initial part of $P$ as defined by Hu and Gao \cite{HuGao2019}. 
We show that these two initial parts carry the same information. 
On one hand, we have the relation 
\[
\In_S(P)=\operatorname{in}_S(P)|_{t=1}.
\]
On the other, given $\In_S(P)$, we can recover $\operatorname{in}_S(P)$ up to a global factor of form $t^i$
by multiplying each monomial by a suitable power of~$t$.
Concretely,
if $\In_S(P)=\sum_{M\in\Lambda}a_M x^M$ for scalars $a_M\in K$, then
\[\operatorname{in}_S(P) = t^i\sum_{M\in\Lambda} t^{\trop(Q)(S)}a_M x^M\]
for some $i\in\mathbb{Z}$.
\end{remark}

  %\end{definition}

\begin{example}\label{l}
Consider the differential polynomial
 $$P(x)=tx_{11}+t^2 x_{13}+ t^3\in K[[t]]\{x_{1}\}.$$
Take the set $S=\{2,3\}\subset\mathbb{Z}_{\geq 0}$.
We obtain $$ \trop(P)(S)=\trop(tx_{11})(S)=\trop(t^2 x_{13})(S)= 2 .$$
 The $ S$-initial part of $ P $ equals
\[
\In_{S}(P)= \overline{t^{-1}tx_{11}}+\overline{t^2t^{-2}x_{13}}=x_{11}+x_{13}\in  K\{x_{1}\}.
\]
Compare this with the $ S$-initial part of $ P $ defined by Hu and Gao:
\[
 \operatorname{in}_S(P)= tx_{11}+t^2 x_{13}\in  K[[t]]\{x_{1}\}.
\]

\end{example}

The next example shows that, in general,
\[
\In_{S}(d(P))\neq d(\In_{S}(P)),
\] 
when $ P\in K[[t]]\{x_{1},\ldots,x_{n}\} $ is a differential polynomial. \begin{example}\label{l2}
Let $P$ be the polynomial of Example~\ref{l}.
The derivative of $ P $ is $$ d(P)= tx_{12}+x_{11}+t^2x_{14}+2tx_{13}+3t^2.$$
 We have
 $$\In_{S}(d(P))=x_{11}+2x_{13}\quad\mbox{and}\quad d(\In_{S}(P))=x_{12}+x_{14}$$
from which we see
$d(\In_{S}(P))\neq \In_{S}(d(P))$.
\end{example}

\begin{lemma}\label{mm}
Let $ P_{1},P_{2}\in K[[t]]\{x_{1},\ldots, x_{n}\} $ be two differential polynomials. We have $$ \In_{S}(P_{1}P_{2})= \In_{S}(P_{1})\In_{S}(P_{2}).$$
\end{lemma}
\begin{proof}
This is true since $ \trop(P_{1}P_{2})(S)= \trop(P_{1})(S)+ \trop(P_{2})(S) $.  
\end{proof}

\begin{obs}\label{rr}
Let $ P\in K[[t]]\{x_{1},\ldots, x_{n}\} $ be a differential polynomial and $ a\in \mathbb{Z} $. We have 
$$ \In_{S}(t^{a}P)=\In_{S}(P) .$$ 
\end{obs}

Let $ I\subset K[[t]]\{x_{1},\ldots,x_{n}\} $ be a differential ideal. Its \newword{$ S $-initial ideal} is 
$$ \In_{S}(I)=\langle \In_{S}(P ) \mid {P\in I}\rangle \subset K\{x_{1},\ldots,x_{n}\}.$$
The next lemma is a more explicit version of \cite[Lemma 2.6]{HuGao2019}.
\begin{lemma}\label{nn}
Let $ I\subset K[[t]]\{x_{1},\ldots,x_{n}\}  $ be a differential ideal. For every $ G\in \In_{S}(I) $ there exists $ g\in I $ such that $  \In_{S}(g)=G$. 
\end{lemma}
\begin{proof}
We can write $ G=\sum \alpha_{M}x^M\In_{S}(G_{M}) $ with $G_M\in I$ for every $M$. 
Take $$ A_{M}=\trop(G_{M})(S)+\sum_{i,j} M_{ij}\Val_{S_{i}}(j), \quad g=\sum_{M} \psi_{M} t^{-A_{M}} x^MG_{M},$$ 
 where $ \overline{\psi_{M}}=\alpha_{M} $ and $ \nu(\psi_{M})=0 $. We calculate
 %Then we have $ \In_{S}(g)=G. $

\begin{align*}
\In_{S}(g)&=\overline{g_{S}(x)} =\overline{t^{-\trop(g)(S)}g (t^{\Val_{S_{i}}(j)}  x_{ij})}
  \\&= \overline{t^{0} \sum_{M} \psi_{M} t^{- A_{M}} \prod_{i,j}(t^{\Val_{S_{i}}(j)}x_{ij})^{M_{i,j}}G_{M}(t^{\Val_{S_{i}}(j)}x_{ij})}
  \\&= \overline{ \sum_{M} \psi_{M} t^{- A_{M}} \Big(\prod_{i,j}(t^{\Val_{S_{i}}(j)})^{M_{i,j}}\Big) x^MG_{M}(t^{\Val_{S_{i}}(j)}x_{ij})}  
  \\&=  \overline{ \sum_{M} \psi_{M} t^{-\trop(G_{M})(S)} x^MG_{M}(t^{\Val_{S_{i}}(j)}x_{ij})}
  \\&=  \sum_{M} \overline{\psi_{M}} x^M \overline{t^{-\trop(G_{M})(S)} G_{M}(t^{\Val_{S_{i}}(j)}x_{ij})}  
  \\&= \sum_{M} \alpha_{M} x^M\In_{S}(G_{M}) .
\end{align*}
We choose $ A=\max_{M\in \Lambda} \{A_{M}\} $ and  take $ H=t^{A}g\in I $.  By Observation \ref{rr} we have 
\[\In_{S}( H)=\In_{S}(t^{A}g)=\In_{S}(g)=G.\qedhere\]    
 \end{proof}

\begin{theorem}\label{pp}
Let $K$ be an uncountable algebraically closed field of characteristic zero. 
Let $ I\subset K[[t]]\{x_{1},\ldots,x_{n}\} $  be a differential ideal. Then the following three subsets of $ \left(\mathcal{P}(\mathbb{Z}_{\geq 0})\right)^n$ coincide: 
\begin{enumerate}
 \item\label{set1} $\trop(\Sol(I))$,     
 \item\label{set2}  $\Sol(\trop(I))$,   
 \item\label{set3} $ A=\{S\subset \left(\mathcal{P}(\mathbb{Z}_{\geq 0})\right)^n\mid \In_{S}(I) $ does not contain a monomial $ \} $.
\end{enumerate}
\end{theorem}
\begin{proof}
Set~(\ref{set1}) is contained in set~(\ref{set2}) by Proposition 5.2 of~\cite{ArocaGarayToghani:2016}. 
We now prove that set~(\ref{set2}) is contained in set~(\ref{set3}). Let $ S\in \Sol(\trop(I)) $. 
Then for any polynomial $ P=\sum_{M\in\Lambda }\varphi_{M}x^M\in I $ the minimum of $\trop(P)(S)= \min\{\sum_{i,j} M_{ij}\Val_{S_{i}}(j)+\nu(\varphi_{M}) \mid M\in \Lambda\} $ is achieved twice or is $ \infty $. 
The terms attaining this minimum are the terms that make up $ \In_{S}(P) $,
so in neither case is $ \In_{S}(P) $ monomial. 
This implies that $\In_{S}(I)$ contains no monomial, 
because if it did, say the monomial $ G\in \In_{S}(I)$, then Lemma~\ref{nn} provides $ H\in I $ such that $\In_{S}(H)=G $.
So the containment of set~2 in set~3 is proved.

Finally, set~(\ref{set3}) is contained in set~(\ref{set1}). 
Let $ S=(S_{1},\ldots,S_{n})\in A $, and suppose $ S\notin \trop(\Sol(I)) $. 
In the proof of Theorem 8.1 in \cite{ArocaGarayToghani:2016}
one exhibits a polynomial $ g $ in $ I$, say
$$ g=\prod_{\substack{0\leq i\leq n, \\ j\in S_i,\, j\leq N_m }}(x_{ij})^{\alpha}
+h+t\lambda \in I$$
where
$h\in \langle x_{ij}, j\notin S_{i},j\leq N_m \rangle$ and $\lambda \in K[[t]]\{x_{0},\ldots,x_{n}\}$,
such that we have $ \trop(g)(S)=0 $ and the minimum in this tropicalisation is attained only at the first monomial on the right hand side. 
Then $ \In_{S}(g) $, the initial part of $ g $, is a monomial. This contradicts $ S\in A $.
\end{proof}

 %Consider a differential polynomial 
 % \begin{equation}\label{polpowfie}
 %     Q=\sum_{M\in \Lambda} \varphi_{M}x^M \in K((t))\{x_{1} , \ldots, x_{n} \},
 % \end{equation}
 % We take $ a=\min\{ \nu(\varphi_{M}) \mid M\in\Lambda\} $,
  %\zeinab{we did not define $m\In_{M}\{ \nu(\varphi_{M})\}$ , also I removed it from previous proof, I thought I wrote wrong.}
 %\begin{enumerate}
 %\item If $ a\geq 0 $ then $ Q\in K[[t]]\{x_{1} , \cdots, x_{n} \}. $ 

 %\end{enumerate}

\section{Nonexistence of certain bases}\label{sec:bases}
This section is dedicated to bases of tropical differential ideals.
We first recall the notion from non-differential tropical geometry which we wish to generalise.
If $L$ is a valued field, then
for each element of the polynomial ring $f\in L[x_{1},\ldots,x_{n}]$ we can define a \newword{tropical hypersurface} $V_{\trop}(f)\subseteq\mathbb{T}^n$.
Then if $I\subseteq L[x_{1},\ldots,x_{n}]$ is an ideal, its \newword{tropical variety} is defined to be
\begin{equation}\label{eq:tropical variety}
V_{\trop}(I) = \bigcap_{f\in I}V_{\trop}(f).
\end{equation}
A \newword{tropical basis} for $I$ is a subset of~$I$ which can replace $I$ as the set over which the intersection is taken in equation~\eqref{eq:tropical variety}.
Crucially, every ideal $I$ has a \emph{finite} tropical basis \cite[Thm~11]{bogart2007computing}.

\subsection{Motivation}
We begin by comparing tropical solutions according to Grigoriev's definition (Definition~\ref{def:tropical solution})
with points of tropical varieties.
Let $ I $ be a nonzero differential ideal in $  K[[t]]\{x_{1},\ldots,x_{n}\} $.
For each natural number $r$
we can define the (non-differential) elimination ideal,
\begin{equation*}
 I_{r}=I \cap  K[[t]][x_{ij}\mid 1\leq i\leq n, 0\leq j\leq r]
\end{equation*}
which we will call a \newword{truncation}.
We form its tropical variety $V_{\trop}(I_{r})$, 
taking the valued field $L$ to be $K((t))$ with its $t$-adic valuation $\nu$.  
If $s\ge r$ is another natural number, then there is a projection map $V_{\trop}(I_{s})\to V_{\trop}(I_{r})$.
Let $V_\infty(I)\subseteq\mathbb{T}^{{n(\mathbb{Z}_{\geq 0})}}$ be the inverse limit of this system of maps.
By the ambient space $\mathbb{T}^{n(\mathbb{Z}_{\geq 0})}$ we mean the Cartesian product of a countable number of copies of~$\mathbb{T}$,
one for each pair of indices $(i,j)$ of a variable $x_{ij}$.

Let $\varphi\in K[[t]]^n$ be a solution of~$I$.
By the Fundamental Theorem of (non-differential) Tropical Geometry, 
the vector $(\nu(d^j\varphi_i) : 1\leq i\leq n, 0\leq j\leq r)$ lies in $V_{\trop}(I_{r})$ for each~$r$.
Therefore $(\nu(d^j\varphi_i) : 1\leq i\leq n, j\geq 0)$ is an element of $V_\infty(I)$.
We warn the reader that the converse does not hold in general!  
See Appendix~\ref{sec:denef} for an example: 
the appendix describes an element $(\nu(d^j\varphi_i) : 1\leq i\leq n, j\geq 0)\in V_\infty(I)$, but $\Sol(I)$ is empty.

Define the map
\begin{displaymath}
\begin{array}{cccc}
\Val:&(\mathcal{P}(\mathbb{Z}_{\geq 0}))^n &\rightarrow & % \mathbb{R}\times 
\mathbb{T}^{n(\mathbb{Z}_{\geq 0})}\\
& S=(S_{1},\ldots,S_{n})&\mapsto & \left(%1, 
\mathrm{Val_{S_{i}}}(j)\right)_{i,j} \ .
\end{array}
\end{displaymath}
The map $\Val$ translates from Grigoriev's set-valued solutions to the infinite tropical vectors above,
so we have
\[\Val(\supp(\varphi)) = (\nu(d^j\varphi_i) : 1\leq i\leq n, j\geq 0)\in V_\infty(I).\]
By Theorem~\ref{pp}, we conclude that $\Val(\Sol(\trop(I)))\subseteq V_\infty(I)$.

Each ideal $I_{r}$ has a finite tropical basis.
Intersection of disjoint polyhedra is subadditive in codimension,
so the cardinality of a tropical basis for $I_{r}$ is not less than its codimension.
By the theory of Kolchin dimension polynomials \cite{Kolchin:1964},
the codimension of $I_{r}$ eventually grows linearly with $r$,
and therefore tropical bases grow at least this fast.

Whatever a basis $B$ for a tropical differential ideal $I$ should be,
one should be able to extract from it a tropical basis for $I_{r}$ for any given $r\in\mathbb N_{\geq 0}$.
By the above discussion,
this means one should be able to produce from $B$ a set of elements of $I_{r}$ of size $\Omega(r)$.
Any polynomial $f\in I$, say of order exactly $m$,
has $r-m$ derivatives $f, df, \ldots, d^{r-m}f$ lying in $I_{r}$, for each $r\ge m$.
This suggests that finite bases in the sense of the following definition may exist.

\begin{definition}\label{def:basis}
In this section, we will call a subset $G$ of $ I $ 
a \newword{tropical differential basis} for $ I $ if 
 \[
  \Sol(\trop(I))= \bigcap_{g\in G}\bigcap_{k=0}^{\infty} \Sol(\trop(d^kg)). 
 \]
\end{definition}

Observe that the containment
 \[
  \Sol(\trop(I))\subseteq \bigcap_{g\in G}\bigcap_{k=0}^{\infty} \Sol(\trop(d^kg)). 
 \]
is true for any $G\subset I$.

Some sufficiently simple differential ideals
do have tropical bases in the sense of Definition~\ref{def:basis}.

\begin{example}
Consider the differential ideal $I=\langle f\rangle_{d} $, where $f=x_{10}+x_{11}+x_{12}$. 
Its solutions are
\[\Sol(I) = \{
  a\big(\sum_{k=0}^\infty t^{3k}-t^{3k+2}\big) 
+ b\big(\sum_{k=0}^\infty t^{3k+1}-t^{3k+2}\big) 
: a,b\in K\}.\]
Theorem 8.1 in \cite{ArocaGarayToghani:2016} implies that $\Sol(\trop(I))=\trop(\Sol(I))$,
from which we can compute $\Sol(\trop(I))$:
\begin{multline*}
    \trop(\Sol(I)) = \{\emptyset,\quad  \mathbb{Z}_{\geq 0},\quad \\ 
    \mathbb{Z}_{\geq 0}\setminus \{ 3k : k\geq 0\},\quad   \mathbb{Z}_{\geq 0}\setminus 
    \{ 3k+1 : k\geq 0\} ,\quad  \mathbb{Z}_{\geq 0}\setminus \{ 3k+2 : k\geq 0\}   \}.
\end{multline*}
Now let $g=f-f'=x_{10}-x_{13}\in I$.
If $S\in \mathcal{P}(\mathbb{Z}_{\geq 0})$ is a solution of 
$\trop(d^kg)=x_{1k}\oplus x_{1(k+3)}$
then $\Val_S(k) = \Val_S(k+3)$.
So if $S$ is any element of $\bigcap_{k=0}^{r} \Sol(\trop(d^kg))$,
then $\Val_S$ is a periodic function on $\mathbb{Z}_{\geq 0}$ of period~3,
implying that $S = \{j-3 : j\in S,j\ge3\}$.
There are $2^3$ such sets~$S$.
It is then just a matter of checking which of these are solutions to each
$\trop(d^kf)=x_{1k}\oplus x_{1(k+1)}\oplus x_{1(k+2)}$ to see that
\begin{multline*}
    \bigcap_{k= 0}^{\infty} \Sol(\trop(d^kf)) \cap \bigcap_{k=0}^{\infty} \Sol(\trop(d^kg))=
  \{\emptyset,\quad  \mathbb{Z}_{\geq 0},\quad  \\
  \mathbb{Z}_{\geq 0}\setminus \{ 3k : k\geq 0\},\quad   \mathbb{Z}_{\geq 0}\setminus 
    \{ 3k+1 : k\geq 0\} ,\quad  \mathbb{Z}_{\geq 0}\setminus \{ 3k+2 : k\geq 0\}   \}
    \\=
\Sol(\trop(I)).
\end{multline*}
We conclude that the set $\{f,g \}$ meets our definition for a tropical differential basis of~$I$.
 %$$\{x_{12}+x_{11}+tx_{10}\}$$ 
\end{example}

 %   \item $\Sol(I)  \subset \lim_{r\rightarrow \infty} V(I_{r})$

   % \item 
    
   % $\Sol(I)  \subset \lim_{r\rightarrow \infty} V(I_{r})$.

\subsection{Linear differential ideals}

This section is dedicated to an example.
Consider the differential ideal $I=\langle x_{12} + sx_{11} + x_{10}\rangle_d\subset K[[t]]\{x_1\}$ 
where $s$ is a very general element of~$K$.
(Any transcendental over~$\mathbb{Q}\subseteq K$ will suffice.
For a more precise statement see \cite[Section 4.4]{FGG}.)
Let $I=\langle f\rangle_d$.  
For any~$r\geq2$, the ideal $I_r$ is generated by the polynomials $d^i(x_{12} + sx_{11} + x_{10})$ for $i=0,1,\ldots,r-2$.
So it is linear in the differential variables.

The tropical variety of a linear ideal $J$
is the Bergman fan of the matroid of its variety $V(J)$, seen as a linear subspace \cite[Theorem 4.1.11]{MaclaganSturmfels:2015}.
In our case, $V(I_r)$ is cut out by the vanishing of the entries of the matrix-vector product 
$A\,(x_{10},x_{11},\ldots,x_{1r})^{\rm T}$,
where $A$ is the $(r+1)\times(r-1)$ band matrix
\[A=\begin{bmatrix}
1 & 0 & \hdots & 0\\
s & 1 &  & 0\\
1 & s &  & 0\\
0 & 1 &  & 0\\
\vdots &  & \ddots & \vdots\\
0 & 0 &  & 1\\
0 & 0 &  & s\\
0 & 0 & \hdots & 1
\end{bmatrix}.\]
By the assumption on $s$, the matrix $A$ has no vanishing $(r-1)\times(r-1)$ minors.
So the matroid of~$V(I_r)$ is the uniform matroid $U_{2,r+1}$ of rank~2 on the elements $\{x_{10}, x_{11}, \ldots, x_{1r}\}$,
whose nonempty proper flats are just its singletons $\{x_{1i}\}$.
The Bergman fan $V_{\trop}(I_{r})\subset\mathbb{T}^{r+1}$ is then the set of all vectors $(a_0, a_1,
\ldots, a_r)$ in which all of the  $a_i$  are equal to some constant $b$
except for perhaps one of them, say  $a_j$,  and we have  $a_j \geq b$
\cite[Definition 4.1.9]{MaclaganSturmfels:2015}.

\begin{proposition}\label{prop:no finite basis}
The above ideal $I$ has no finite tropical differential basis of linear forms in the sense of Definition~\ref{def:basis}.
\end{proposition}

For a nonzero linear form $f = \sum_{j=0}^\infty f_jx_{1j}$, let us write 
\[\suppmin f := \operatorname{argmin}_{j=0}^\infty \nu(f_j)\]
for the set of indices of coefficients of minimal valuation in~$f$.

\begin{lemma}\label{lem:suppmin}
Let $f\in K[[t]]\{x_1\}$ be nonzero.
Suppose that there exists $N\in\mathbb{N}$ such that $|\suppmin(d^kf)|\le N$ for infinitely many $k\in\mathbb{N}$.
Then there exists a finite set $L\subset\mathbb{Z}$ such that 
\[\suppmin(d^kf) = \{k+l : l\in L\}\]
for all but finitely many $k\in\mathbb{N}$.
\end{lemma}

\begin{proof}
The summand of
\[d^j(t^ix_{10}) = \sum_{k=0}^j\frac{i!\,j!}{k!\,(j-k)!\,(i-j+k)!}t^{i-j+k}x_{1k}\]
of largest differential order is a nonzero scalar times $t^ix_{1j}$, and these summands are linearly dependent as $(i,j)$ vary, 
so $f$ can be written uniquely as
\[f = \sum_{i=0}^\infty\sum_{j=0}^{j_*}c_{i,j}d^j(t^ix_{10})\]
for scalars $c_{i,j}\in K$.
There are no issues of convergence in~$K$,
since each scalar coefficient of each power series in~$f$ is a linear combination of finitely many $c_{i,j}$.
Then in
\[d^kf = \sum_{i=0}^\infty\sum_{j=0}^{j_*}c_{ij}d^{j+k}(t^ix_{10}),\]
the $K[[t]]$-valued coefficient of $x_{1(k+l)}$ for any $l\in\mathbb{N}$ has constant term
\[\sum_{j=0}^{j_*}\frac{(j+k)!}{(k+l)!}c_{j-l,j}.\]
Regarding $l$ as fixed, this sum lies in the field of rational functions $K(k)$,
so if it is nonzero for any~$k$ then it is zero for a finite number of~$k$.
Therefore, the $K[[t]]$-valued coefficient of $x_{1(k+l)}$ in $d^kf$ has valuation zero
for all $l$ in
\[L = \{l\in\mathbb{N} : c_{j-l,j}\ne0\mbox{ for some } j\in\mathbb{N}\},\]
for all but finitely many $k\in\mathbb{N}$.  
This implies the lemma for the above choice of~$L$,
which is a finite set by the assumption that $\suppmin(d^kf)\le N$ infinitely often.
\end{proof}

\begin{proof}[Proof of Proposition~\ref{prop:no finite basis}]
Suppose towards a contradiction that $G$ is a finite tropical differential basis of linear forms for~$I$.
For convenience, write $d^*G := \{d^kg : k\in\mathbb{N}, g\in G\}$.

We show that if $f\in I$ is nonzero, then $|\suppmin  f|\ge3$.
Let $l=\min\{\nu(f_j):j\in\mathbb{N}\}$
and choose $r$ to be the differential order of~$f$.
Then $f$ can be written $vA\,(x_{10},x_{11},\ldots,x_{1r})^{\rm T}$ for some row vector $v\in K[[t]]^{r-1}$.
If $|\suppmin f|\le2$, then $\suppmin f$ is disjoint from a set $S\subseteq\{0,\ldots,r\}$ of size $r-1$,
so that $f_j\in t^{l+1}K[[t]]$ for all $j\in S$.
Since the submatrix of~$A$ on row set~$S$ is an invertible matrix over~$K$,
this implies that all entries of $v$ lie in $t^{l+1}K[[t]]$.
But then all coefficients of $f$ lie in $t^{l+1}K[[t]]$ as well, contradicting the choice of~$l$.

Consider the sets
\[
G_r = \big\{f\in d^*G : |\suppmin (f)| = 3,
\suppmin (f)\cap\{0,\ldots,r\}\ne\emptyset\big\}.
\]
For a fixed $g\in G$,
if $|\suppmin(d^kg)|=3$ infinitely often,
then by Lemma~\ref{lem:suppmin} there is a set $L$, in this case necessarily of size~3,
such that $\suppmin(d^kg) = \{k+l:l\in L\}$ for all but finitely many $k$. 
So, aside from a finite number of exceptions,
the derivatives $d^kg$ cease to lie in~$G_r$ once $k+\min L>r$.
If $|\suppmin(d^kg)|=3$ only finitely often then of course $d^kg\in G_r$ only finitely often.
It follows that $|G_r| = O(r)$.

For each pair of nonnegative integers $a,b$ with $b-a\ge2$, 
let  $q_{ab}\in\mathbb{T}^{\mathbb N}$  be the point all of whose coordinates are 0
except that its  $a$\/th and  $b$\/th coordinates are~1.
Observe that $q_{ab}$ is in the image of $\Val$, 
and the truncation $((q_{ab})_j : j=0,\ldots,r)$ lies outside $V_{\trop}(I_{r})$ for any $r\ge b$.
So $G$ must provide a witness that $q_{ab}$ is not a solution of~$I$,
namely a polynomial $f=\sum f_jx_{1j}\in d^*G$ such that $q_{ab}\not\in V_{\trop}(f)$.
Since $|\suppmin(f)|\ge3$, there must be some $j\in\suppmin(f)\setminus\{a,b\}$. 
Then $\nu(f_j)+(q_{ab})_j=\nu(f_j)+0$ is a term of $\trop(f)(q_{ab})$ attaining the minimum.
In order for $\trop(f)(q_{ab})$ not to vanish, there must be only one such $j$,
that is, $\suppmin(f) = \{a,b,j\}$.  
Therefore, 
$\{\suppmin(f) : f\in G_r\}$ contains at least one superset
of each of the $\binom{r+1}2-r$ sets $\{a,b\}\subset\{0, \ldots,r\}$ with $b-a\ge2$.  
Since each 3-element set is a superset of only three 2-element sets, it follows that
\[|G_r|\ge|\{\suppmin (f) : f\in G_r\}|\ge\frac{\binom{r+1}2-r}3.\]
This is a contradiction.
\end{proof}

\appendix

\section{A differential ideal with tropical solutions but no solutions over a field}\label{sec:denef}\label{sec:valuation norm}
In this appendix we repeat an example due to Denef and Lipshitz \cite[Remark 2.12]{DenefLipshitz}.
The example is a differential ideal with a tropical solution and with no solutions over a field $K$ of characteristic zero,
and such that the nonexistence of the latter solutions cannot be detected by truncation to any finite order.
The field $K$ is not algebraically closed,
so that \cite[Theorem 8.1]{ArocaGarayToghani:2016} fails to apply,
but the example is of interest in that the obstruction to lifting the tropical solution
is not the simple unavailability of a coefficient in~$K$ satisfying a certain polynomial equation.

\begin{example}
We begin by considering the differential ideal 
\begin{equation}
    I=\langle tx_{11}-(x_{20}+t)x_{10}-1, x_{21}\rangle_d\subseteq K[[t]]\{x_1,x_2\}.
\end{equation}
The differential polynomial $x_{21}$ vanishes at $x=\varphi\in K[[t]]^2$ only when $\varphi_2$ is a scalar.
In this case solving for the coefficients of $\varphi_1$ recursively shows that its only possible value is
\begin{equation}\label{eq:x(alpha)}
\varphi_1 = \sum_{j=0}^\infty\frac{t^j}{\prod_{k=0}^j(k-\varphi_2)}.
\end{equation}
This $\varphi_1$ is a well-defined element of~$K[[t]]$ unless one of its coefficients features a division by zero.
That is, $I$ has a solution if and only if the scalar $\varphi_2$ lies in $K\setminus \mathbb{N}$
under the canonical inclusion $\mathbb{N}\subseteq K$.

We now specify $K=k(s)$, where $k$ is a formally real field of characteristic zero,
that is, where $-1$ is not a sum of squares in $k$.
Then $\mathbb{N}$ is a polynomially definable set over~$K$
\cite{denef1978diophantine}. 
Let $J\subseteq K[x_2,\ldots,x_n]$ be an ideal such that the projection of $V(J)\subset K^{r-1}$ onto the $x_2$ coordinate is $\mathbb N$.  Then the solution set of the differential ideal
\[J'=\langle J\rangle_d+\langle x_{21},\ldots,x_{n1}\rangle_d\]
is $V(J)\subset K[[t]]^{n-1}$,
the equations $x_{i1}$ constraining each coordinate to be a scalar.

It follows that $\Sol(I+J')=\Sol(I)\cap\Sol(J')=\emptyset$,
because no value of $\varphi_2$ is consistent both with $I$ and with $J'$.

However, for any $r\in\mathbb{Z}_{\geq 0}$,  
any truncation $(I+J')_r$ has solutions $\varphi\in K[[t]]^n$,
one for each natural value of $\varphi_2$ strictly greater than $r$,
since then all of the coefficients in \eqref{eq:x(alpha)} are well-defined elements of~$K$.
The usual Fundamental Theorem of Tropical Geometry implies that 
$V_{\trop}(I_r)$ contains $(\nu(d^j\varphi_i) : 1\leq i\leq n, 0\leq j\leq r)$.
The series \eqref{eq:x(alpha)} has no coefficients equal to zero, and $\varphi_i$ must be a scalar for $i\ge2$,
implying that $\nu(d^j\varphi_1)=0$ for all $j\ge0$
and $\nu(d^j\varphi_i)=\infty$ for all $i=2,\ldots,n$ and $j\ge1$,
while $\nu(d^0\varphi_i)\in\{0,\infty\}$ for $i=2,\ldots,n$.
There are finitely many choices for these last valuations.  Therefore there exists an $n$-tuple of sets $S=(\mathbb Z_{\geq 0},S_2,\ldots,S_n)$,
where $S_i\subseteq\{0\}$ for $i=2,\ldots,n$,
so that 
\[(\Val_{S_i}(j) : 1\leq i\leq n, 0\leq j\leq r) = (\nu(d^j\varphi_i) : 1\leq i\leq n, 0\leq j\leq r)\] 
for infinitely many $r$.
This implies that $S\in\Sol(\trop(I+J'))$.
\end{example}

\bibliographystyle{amsalpha}
\addcontentsline{toc}{chapter}{Bibliografia}
\bibliography{bibliografia_zeinab1}

\newcommand{\etalchar}[1]{$^{#1}$}
\providecommand{\bysame}{\leavevmode\hbox to3em{\hrulefill}\thinspace}
\providecommand{\MR}{\relax\ifhmode\unskip\space\fi MR }
% \MRhref is called by the amsart/book/proc definition of \MR.
\providecommand{\MRhref}[2]{%
  \href{http://www.ams.org/mathscinet-getitem?mr=#1}{#2}
}
\providecommand{\href}[2]{#2}
\begin{thebibliography}{BJS{\etalchar{+}}07}

\bibitem[AGT16]{ArocaGarayToghani:2016}
Fuensanta Aroca, Cristhian Garay, and Zeinab Toghani, \emph{The fundamental
  theorem of tropical differential algebraic geometry}, Pacific Journal of
  Mathematics \textbf{283} (2016), no.~2, 257--270.

\bibitem[BJS{\etalchar{+}}07]{bogart2007computing}
Tristram Bogart, Anders~Nedergaard Jensen, David Speyer, Bernd Sturmfels, and
  Rekha~R Thomas, \emph{Computing tropical varieties}, Journal of Symbolic
  Computation \textbf{42} (2007), no.~1-2, 54--73.

\bibitem[Den78]{denef1978diophantine}
Jan Denef, \emph{The diophantine problem for polynomial rings and fields of
  rational functions}, Transactions of the American Mathematical Society
  \textbf{242} (1978), 391--399.

\bibitem[DL84]{DenefLipshitz}
Jan Denef and Leonard Lipshitz, \emph{Power series solutions of algebraic
  differential equations}, Mathematische Annalen \textbf{267} (1984), no.~2,
  213--238.

\bibitem[FGG]{FGG}
Alex Fink, Jeffrey Giansiracusa, and Noah Giansiracusa, \emph{Projective
  hypersurfaces in tropical scheme theory}, with appendix by Joshua Mundinger.
  In preparation.

\bibitem[Gri17]{Grigoriev:2015}
Dima Grigoriev, \emph{Tropical differential equations}, Advances in Applied
  Mathematics \textbf{82} (2017), 120--128.

\bibitem[HG19]{HuGao2019}
Youren Hu and Xiao-Shan Gao, \emph{Tropical differential {G}roebner basis},
  arXiv preprint arXiv:1904.02275 (2019).

\bibitem[Kol64]{Kolchin:1964}
ER~Kolchin, \emph{The notion of dimension in the theory of algebraic
  differential equations}, Bulletin of the American Mathematical Society
  \textbf{70} (1964), no.~4, 570--573.

\bibitem[MS15]{MaclaganSturmfels:2015}
Diane Maclagan and Bernd Sturmfels, \emph{Introduction to tropical geometry},
  Graduate Studies in Mathematics, vol. 161, American Mathematical Society,
  Providence, RI, 2015. \MR{3287221}

\end{thebibliography}

\end{document}